\documentclass[reqno]{amsart}

\usepackage{amsmath,amssymb,color}
\usepackage{amsthm}
\usepackage{enumerate,enumitem}
\usepackage{graphicx,subfigure}
\usepackage{verbatim,mathrsfs}
\usepackage{cite}

\DeclareGraphicsExtensions{.pdf,.png,.eps}
\graphicspath{{figs/}}

\newtheorem{theorem}{Theorem}

\newtheorem{lemma}[theorem]{Lemma}
\newcommand{\cM}{\mathcal{M}}
\newcommand{\cH}{\mathcal{H}}

\newcommand{\cE}{\mathcal{E}}

\begin{document}

\title{High girth hypergraphs with unavoidable monochromatic or rainbow edges}
\author[M.~Axenovich]{Maria Axenovich}
\author[A.~Karrer]{Annette Karrer}

\address[M.~Axenovich, A.~Karrer]{Karlsruher Institut f\"ur Technologie, Karlsruhe, Germany}
\email{maria.aksenovich@kit.edu}

  \date{\today}
 \maketitle

  \begin{abstract}
 A classical result of Erd\H{o}s and Hajnal claims that for any integers $k, r, g \geq 2$ there is an $r$-uniform hypergraph of girth at least $g$ with chromatic number at least $k$. This implies that there are  sparse hypergraphs such that in any coloring of their vertices with at most $k-1$ colors there is a monochromatic hyperedge. 
 We show that for any integers $r, g\geq 2$ there is an  $r$-uniform hypergraph of girth at least $g$ such that in any coloring of its vertices there is either a monochromatic or a rainbow (totally multicolored) edge.
 We give a probabilistic and a deterministic proof of this result.
  \end{abstract}

 \section{Introduction}

   A classical result of Erd\H{o}s and Hajnal \cite{erd66}, Corollary 13.4, claims that for any integers $k, r, g \geq 2$ there is an $r$-uniform hypergraph of girth at least $g$ with chromatic number at least $k$. This implies that there are  sparse hypergraphs such that in any coloring of their vertices with at most $k-1$ colors there is a monochromatic hyperedge.  The original proof was probabilistic.  Other probabilistic constructions were given by Ne\v{s}et\v{r}il and R\"odl   \cite{nese78},  Duffus et al. \cite{duffus},  Kostochka and R\"odl   \cite{kost10}, and, in case of graphs only, by  Erd\H{o}s  \cite{erd59}.
  Several explicit constructions were found later, see   Lov\'asz   \cite{lov68},    Erd\H{o}s and Lov\'asz   \cite{erd75}, 
  Ne\v{s}et\v{r}il and R\"odl  \cite{nese79},  Duffus  et al.   \cite{duffus},  Alon et al.  \cite{alon},  K\v{r}\'i\v{z} \cite{kriz89}, Kostochka and  Ne\v{s}et\v{r}il \cite{kost99}. 
Ne\v{s}et\v{r}il  \cite{nese13} as well as  
Raigorodskii and Shabanov \cite{raig11} gave surveys on the topic.   Some interesting generalizations were treated by Feder and Vardi \cite{feder98},
Kun \cite{kun}, M\"uller     \cite{muller75} , \cite{muller79},  as well as by  Ne\v{s}et\v{r}il  \cite{nese13}.

 When the number of colors used on the vertices of a hypergraph is not restricted, the monochromatic hyperedges could easily be avoided by simply using a lot of  different  colors. Then, however,  so-called rainbow (totally multicolored) hyperedges could appear.    The notion of a proper coloring when both rainbow and monochromatic hyperedges are forbidden was introduced by Voloshin in a concept called bihypergraphs, \cite{V}, see also Karrer \cite{karrer}.
 Here, we show that  there are  sparse hypergraphs in which monochromatic or rainbow hyperedges are unavoidable.
   
 {\it A cycle}  of length $g$ in a hypergraph is a subhypergraph consisting of  $g\geq 2$  distinct  hyperedges $E_0, \ldots, E_{g-1}$  and containing distinct vertices  $x_0, \ldots, x_{g-1}$,  such that $x_i  \in E_i\cap E_{i+1}$, $i=0, \ldots, g-1$, addition of indices modulo $g$.    The {\it  girth} of a hypergraph is the length of a shortest cycle if such exists, and infinity otherwise. Next is our main result.

 \begin{theorem}\label{main}
 For any integers $r, g\geq 2$ there is an  $r$-uniform hypergraph of girth at least $g$ such that in any coloring of its vertices there is either a monochromatic or a rainbow (totally multicolored) edge.
 \end{theorem}

   We shall give a probabilistic proof and an explicit construction of a desired hypergraph. Our proofs are inspired by amalgamation and probabilistic techniques of  Ne\v{s}et\v{r}il  and  R\"odl.  
To shorten the presentation, we shall say that a hypergraph is {\it rm-unavoidable} if any coloring of its vertices has either a rainbow or a 
monochromatic edge.
We give an explicit construction and use it to prove the main theorem in Section \ref{explicit}. The probabilistic proof is given in Section \ref{probabilistic}.  The proofs of a  few standard results we use are presented in Appendix.

\section{Explicit Construction of rm-unavoidable Hypergraphs}\label{explicit}
The goal of this section is to construct, for each $r\geq 2$ and $g\geq 2$, an  rm-unavoidable hypergraph,  that we shall call $H(r,g)$, of uniformity $r$ and girth $g$.
The three main concepts we use are amalgamation,  special partite hypergraphs forcing rainbow edges, and so-called complete partite factors. All of these notions are defined for partite hypergraphs. A hypergraph is {\it $a$-partite}  if its vertex set can be partitioned in at most  $a$ parts such that each hyperedge contains at most one vertex from each part. We shall first define a part-rainbow-forced hypergraph as a hypergraph having some special coloring properties and give an explicit construction of 
such a hypergraph $PR(r,g)$. Then we incorporate this hypergraph into a more involved  construction of an rm-unavoidable hypergraph $H(r,g)$. Both of these constructions use amalgamation.\\

\begin{figure}[htb] 
\includegraphics[scale=0.8]{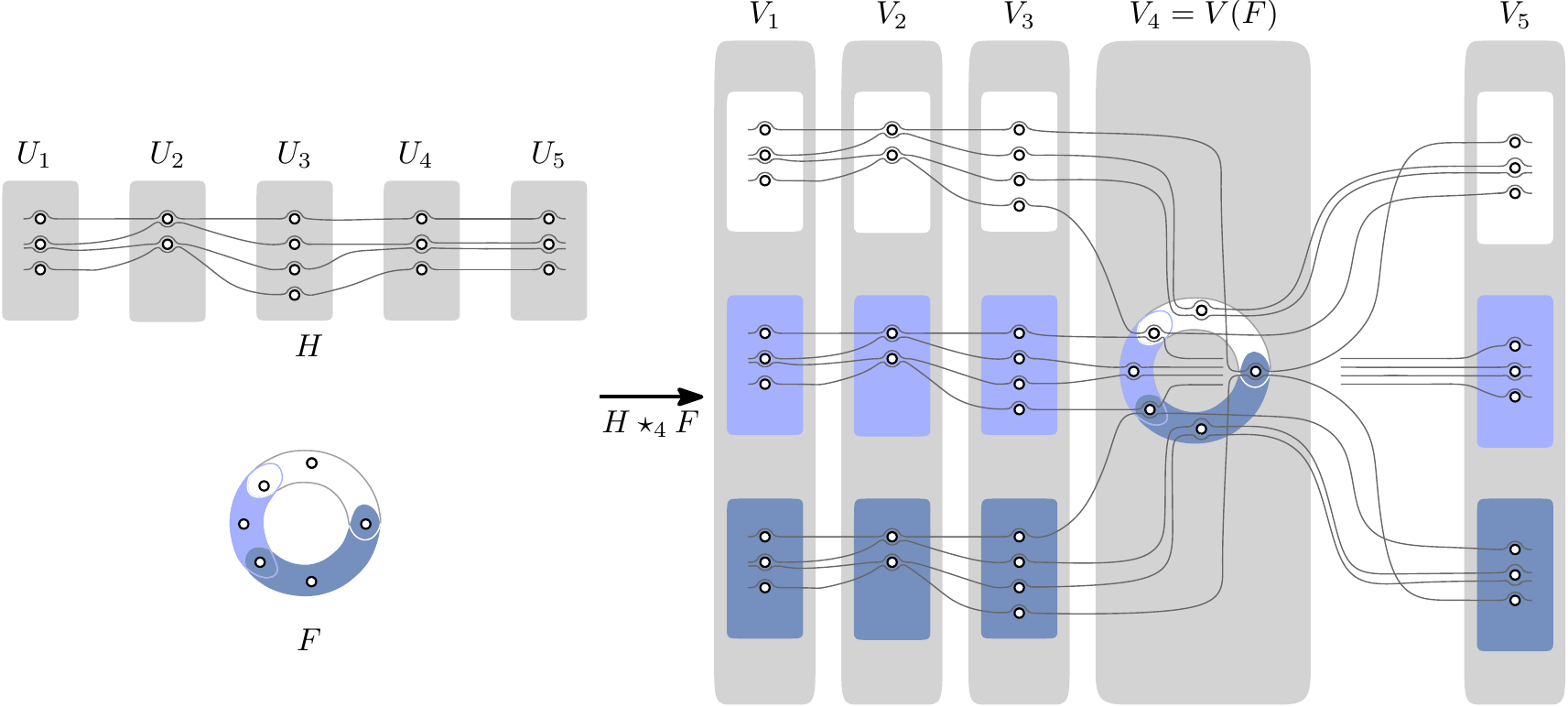}
\caption{Amalgamation of $F$ and $H$ along the $4^{{\rm th}}$ part. 
Here $F$ is a $3$-uniform cycle on $3$ edges, $H$ is $5$-uniform, $5$-partite with $4$ edges. 
The resulting graph is $5$-partite, $5$-uniform, with curves indicating hyperedges and colors indicating distinct copies of $H$, corresponding to the edges of $F$.} \label{amalgamation}
\end{figure}

\begin{description}
\item[Amalgamation]
Given an $a$-partite hypergraph $H$ with  the $i^{{\rm th}}$ part of size $r_i$ and given an  $r_i$-uniform hypergraph $F=(V, \cE)$,
{\it an amalgamation} of $H$ and $F$ along  the $i ^{{\rm th}}$ part, denoted by $H\star_i  F$ is an $a$-partite hypergraph obtained by taking $|\cE|$ vertex-disjoint copies of $H$ and identifying the $i$th part of each such copy with a hyperedge of $F$ such that distinct copies get identified with distinct hyperedges.
Moreover, the $j^{\rm th}$ part of  $H\star_i  F$  is a  pairwise disjoint union of the $j^{{\rm  th}}$  parts from the copies of $H$, for $j\in \{1, \ldots, a\} \setminus \{i\}$, see Figure \ref{amalgamation}.  We shall sometimes say that $H\star_i  F$ is obtained by amalgamating copies of $H$  along the  part $i$ using $F$.\\

\item[Part-rainbow-forced hypergraph]
A vertex coloring of an $a$-partite hypergraph with parts $X_1, \ldots, X_a$ that 
assign $|X_i|$ colors to part $i$, $i=1, \ldots, a$ is called {\it part rainbow}.
We say that an $a$-partite  hypergraph  is {\it part-rainbow-forced} if in any  part-rainbow coloring  there is a rainbow edge.\\

\begin{figure}[htb] 
\includegraphics[scale=0.8]{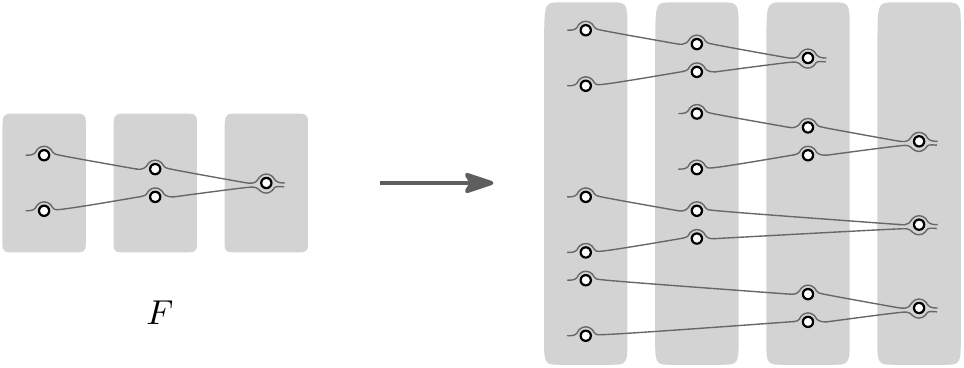}
\caption{An example of a complete $4$-partite $F$-factor, where $F$ is a $3$-partite $3$-uniform hypergraph with two edges.}\label{partite-factor}
\end{figure}

\item[Partite factor]
Let $F$ be an $r$-uniform $r$-partite hypergraph.  A {\it complete  $a$-partite $F$-factor} is an $a$-partite $r$-uniform  hypergraph $G$ that is a  union of   pairwise vertex-disjoint copies $F_1, \ldots, F_{\binom{a}{r}}$ of $F$, such that  each part  of $F_i$ is contained in some part of $G$, $i = 1, \ldots, \binom{a}{r}$ and such that 
the union of any  $r$ parts of $G$   contains the vertex set of  $F_i$, for some $i=1, \ldots, \binom{a}{r}$, see Figure \ref{partite-factor}.\\

 \begin{figure}[htb]
\includegraphics[scale=1.0]{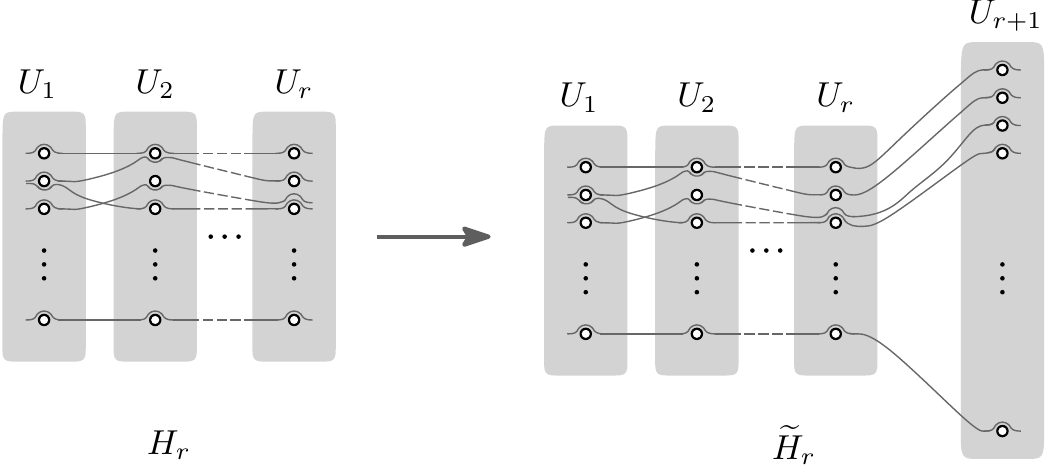}
\caption{Extension of an  $r$-partite  $r$-uniform hypergraph $H_r$ to an $(r+1)$-partite  $(r+1)$-uniform hypergraph $\widetilde{H}_{r}$.} \label{extension}
\end{figure}

\item[Construction of a hypergraph \textbf{\textit{PR(r,g)}}] ~~
Let  $r, g \geq 2$, $g\geq 2$ be fixed.  
Let  $g\geq 2$, let $PR(2, g)$ be a bipartite graph on vertices $x,y,z$ and edges $xy$,  $yz$. 

Assume now that  $PR(r,g)$ has been constructed and it is  an $r$-uniform, $r$-partite hypergraph. 
Let $F' $ be an $\ell$-uniform hypergraph  of girth at least $g$ and minimum degree $\ell (r+1)$, where $\ell = |E(H_r)|$.   We show the existence of $F'$  in Appendix.

For an $r$-uniform $r$-partite hypergraph $H$, let $\widetilde{H}$ be an $(r+1)$-partite  $(r+1)$-uniform hypergraph that is obtained from $H$ by expanding each of its edges by a vertex in a new, $(r+1)^{{\rm st} }$ part such that 
each edge is extended by an own vertex, i.e., the size of the $(r+1)^{{\rm st} }$ part is equal to the number of edges in $H$, see Figure \ref{extension}.

Let $PR(r+1, g) = \widetilde{PR(r, g)} \star _{r+1} F'$, i.e., it is an amalgamation of copies of $\widetilde{PR(r,g)}$ along the $(r+1)^{{\rm st}}$ part using $F'$, see Figure \ref{part-rainbow}.
 \end{description}
~\\

 \begin{figure}[htb] 
\includegraphics[scale=0.85]{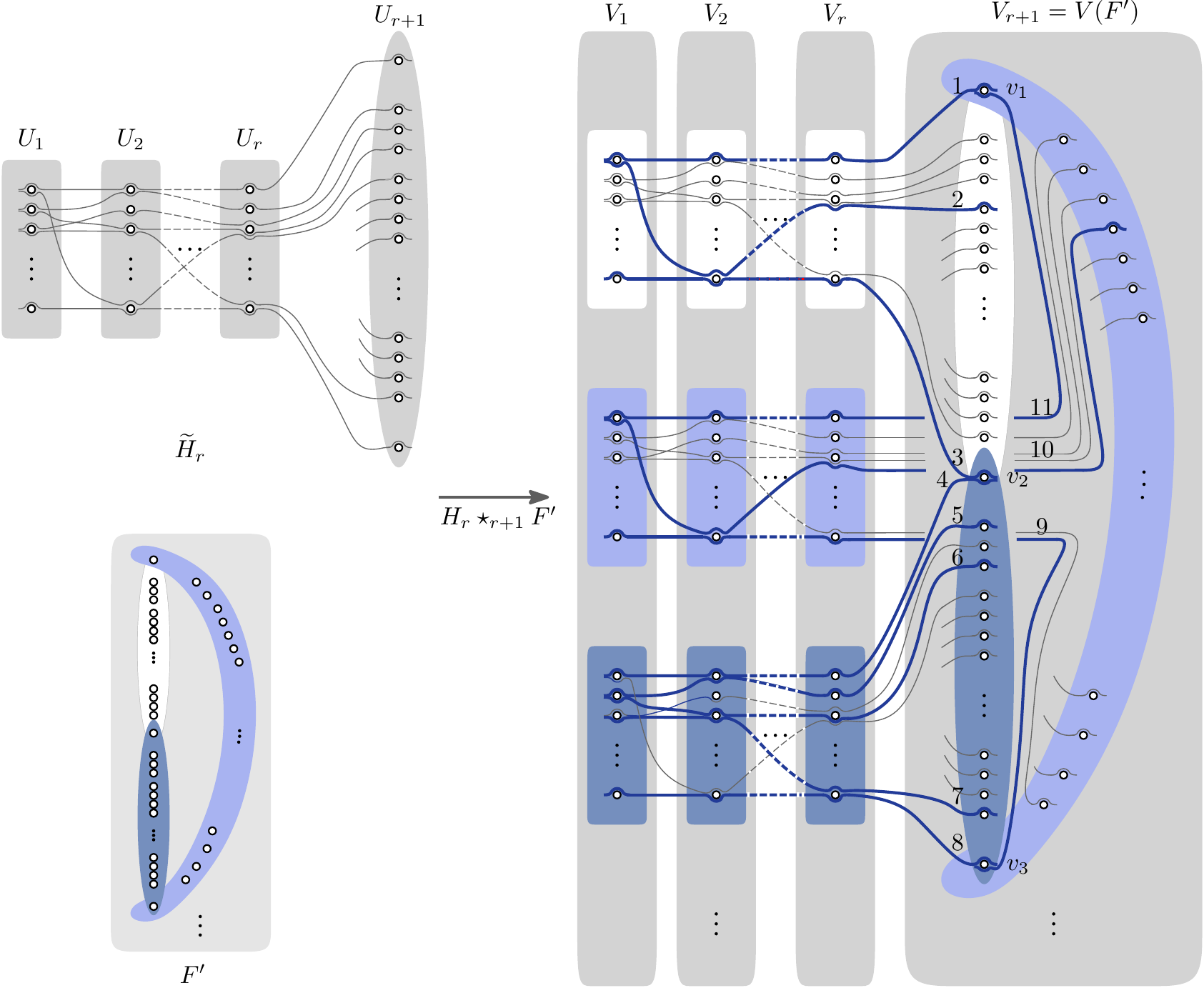}
\caption{Illustration of a part-rainbow-forced  $(r+1)$-uniform hypergraph and a cycle of length $3$ in the amalgamated hypergraph $F'$.  The bold  hyperedges form a cycle of length $11$ in the resulted hypergraph.}\label{part-rainbow}
\end{figure}

 \begin{lemma}\label{Lem2}
 For any integers $r, g\geq 2$,  $PR(r,g)$ is  a part-rainbow-forced $r$-uniform hypergraph of girth $g$.
 \end{lemma}

 \begin{proof}
By construction,  $PR(r,g)$ is an $r$-uniform $r$-partite hypergraph, $r\geq 2$.
We shall prove by induction on $r$ that $PR(r,g)$ is  part-rainbow-forced hypergraph of girth at least $g$.
When $r=2$, we see that a part-rainbow coloring assigns distinct colors to $x$ and $z$. 
Thus, no matter how $y$ is colored, $xy$ or $yz$ is rainbow. Moreover this graph is acyclic, so it has infinite girth.

Assume that $PR(r,g)$ is  part-rainbow-forced hypergraph of girth at least $g$. 
Let's prove that $H_{r+1} = PR(r+1,g)$ is  also part-rainbow-forced hypergraph of girth at least $g$.   Let $H_r= PR(r, g)$.
Recall that $H_{r+1}$ is an  amalgamation of copies  $\widetilde{H}_r^1$,  $\widetilde{H}_r^1$, \ldots, $\widetilde{H}_r^{e'}$  of $\widetilde{H_r}$ along the $(r+1)^{{\rm st}}$ part using $F'$, where $F' $ is an $\ell$-uniform hypergraph  of girth at least $g$, minimum degree $\ell (r+1)$,  $\ell = |E(H_r)|$,  and $e' = |E(F')|$. 
Recall further, that  $\widetilde{H}_r^i$ is obtained by an extension  operation tilde from  $H_r^i$, a copy of $H_r$.

First we shall verify that any part-rainbow coloring $c$ of $H_{r+1}$ results in a rainbow edge.
For any $i=1, \ldots, e'$, consider a restriction of $c$ to the vertex set of $H_r^i$. Since it is a copy of $H_r = PR(r,g)$,   it is again part-rainbow, so there is a rainbow edge $E'_i$ in that copy.  Let $ E'_i\cup \{v_i\}$ be a corresponding uniquely defined edge of $\widetilde{H}^i_{r}$. 
The vertices $v_1, \ldots, v_{e'}$  are vertices of $F'$. 
Since the minimum degree of $F'$ is at least $\ell (r+1)$, then $e'=|E(F')|\geq  |V(F')| \ell(r+1) / \ell = |V(F')| (r+1)$. 
Thus there are at least $r+1$ repeated vertices in the list $v_1, \ldots, v_{e'}$, i.e.,  w.l.o.g.  $v=v_1= \ldots = v_{r+1}$. 
Thus $v$ extends rainbow edges  $E'_1, E_2', \ldots, E_{r+1}'$  in $H_r^1, H_r^2, \ldots, H_r^{r+1}$.   
We claim that at least one of the extended edges  $E'_1\cup \{v\}, E_2'\cup \{v\}, \ldots, E_{r+1}'\cup\{v\}$ is rainbow. 
Assume not,  then $c(v)$ is present in each of $E'_1, E_2', \ldots, E_{r+1}'$. However, there are at most $r$ vertices of each given color in the first $r$ parts.  Since $E'_1, E_2', \ldots, E_{r+1}'$ are pairwise disjoint, we have a contradiction.

To see that the  girth of $H_{r+1}$  is at  least $g$, consider a  cycle $C$ in $H_{r+1}$, see bold edges in Figure \ref{part-rainbow}. If the edges of $C$ come from  one copy of $\widetilde{H}_r$, then  the length of $C$ is at least $g$ as the girth of $\widetilde{H}_r$ is the same as girth of $H_r$.   If the  edges of $C$  come from at least two distinct  copies of  $\widetilde{H}_r$,  then $C$ is  a  union of hyperpaths  $P_0, P_1, \ldots, P_{m-1}$ from different copies of $\widetilde{H}_r$, such that the consecutive paths  share a vertex in the last $(r+1)^{\rm st}$ part, i.e., $V(P_i)\cap V(P_{i+1}) = \{u_i\}$, 
$u_0, \ldots, u_{m-1}$ are distinct vertices from $V_{r+1}$, addition modulo $m$.
Thus $u_{i}$ and $u_{i+1}$ belong to the same copy of $\widetilde{H}_r$ and thus the same edge of $F'$, $i=0, \ldots, m-1$, addition modulo $m$.
 We see that these  edges of $F'$ form a cycle in $F'$ of length at most the length of $C$. 
 On the other hand, we know that any cycle in $F'$ has length at least $g$, implying that  $C$ has length at least $g$. This concludes the proof that $PR(r+1, g)$ is part-rainbow-forced of girth at least~$g$.     \end{proof}

Now we construct an  rm-unavoidable hypergraph $H(r,g)$ of uniformty $r$ and girth at least~$g$.\\

\begin{description}
\item[Construction of a hypergraph \textbf{\textit{H(r, g)}}]  ~~
For $g=2$ and any $r\geq 2$,  let $H(r, 2)$ be a complete $r$-uniform hypergraph on $(r-1)^2+1$ vertices.  
Assume that  for any $r\geq 2$,  $H(r,g-1)$ has been constructed. 
Let $F= PR(r, g)$ be  as given in the previous construction.  Let $a= (r-1)^2 +r$ and let $\cM_1$ be a complete $a$-partite  $F$-factor.
For any partite hypergraph $G$, let $|G|_i$ denote the size of the  $i^{{\rm th}}$ part of $G$.\\
Let $\cM_2 = \cM_1 \star_1 \cH_1$, where $\cH_1= H(|\cM_1|_1, g-1)$.  
Let $\cM_3= \cM_2 \star_2 \cH_2$,  where $\cH_2=H(|\cM_2|_2, g-1)$. 
In general, let $\cM_{j+1}= \cM_j \star_j \cH_j$,  where $\cH_j=H(|\cM_j|_j,g-1)$. 
We see that the $j^{\rm th}$ part of $\cM_{j+1}$ corresponds to the  vertex set of $\cH_j$.
Let $H(r,g)=\cM_{a+1}$.
\end{description}

\noindent
Now, we shall prove that this construction gives an rm-unavoidable hypergraph that is $r$-uniform and has girth $g$.  This will give  a proof of Theorem \ref{main}.

\begin{proof}[Proof of Theorem \ref{main}]
We shall show that $H(r, g)$ is an rm-unavoidable hypergraph of girth at least $g$, by induction on $g$.
When $g=2$, $H(r, 2)$  is a compete $r$-uniform hypergraph on $(r-1)^2 +1$ edges.  It has girth $2$ and in any vertex coloring there are either $r$ vertices of the same color, forming a monochromatic edge, or $r$ vertices of distinct colors, forming a rainbow edge.
Assume that for any $r\geq 2$,  $H(r, g-1)$  is an rm-unavoidable hypergraph of girth at least $g-1$.

Consider $H(r, g)  = \cM = \cM_{a+1}$ given in the construction.
Let  $c$ be a vertex coloring of $ \cM $. Consider the $a^{\rm th}$ part of $\cM=\cM_{a+1}$. 
This part corresponds to the vertex set of $\cH_a = H(|\cM_a|_a, g-1)$, an rm-unavoidable hypergraph.
Thus, there is a monochromatic or rainbow  subset $X_a$  in the $a^{\rm th}$  part of $\cM$ of size equal to the uniformity of $\cH_a$, i.e., of size $|\cM_a|_a$. 
Since  $X_{a}\in \cE(\cH_{a})$, $X_a$ is the $a^{\rm th}$ part of a copy of $\cM_a$.

Consider $(a-1)^{\rm st}$ part of this copy of $\cM_a$. Similarly to the above, there is a  monochromatic or rainbow  subset $X_{a-1}$ of this part 
of size equal to the uniformity of  $\cH_{a-1}= H(|\cM_{a-1}|_{a-1}, g-1)$, i.e., of size $|\cM_{a-1}|_{a-1}$.
Since  $X_{a-1}\in \cE(\cH_{a-1})$,  $X_{a-1}$  is the $(a-1)^{\rm st}$ part of  a copy of $\cM_{a-1}$ such that the $a^{\rm th}$ part of this copy is a subset of $X_a$.

Continuing in this manner we see that there is a monochromatic or a rainbow  subset $X_{j}$ of $j^{\rm th}$  part of $\cM_{j+1}$ 
of size equal to the uniformity of  $\cH_{j}$, i.e., of size $|\cM_{j}|_{j}$.
We have that  $X_{j}$ is the $j^{\rm th}$ part of  a copy of $\cM_{j}$ such that the $(j+t)^{\rm th}$ part of this copy is a subset of $X_{j+t}$, 
$j+t \in \{j+1, j+2, \ldots, a\}$.

Thus $X_1, X_2, \ldots, X_a$ form parts of an $a$-uniform sub-hypergraph of $\cM$ containing a copy of $\cM_1$. Recall that $\cM_1$  is a complete $a$-partite  $F$-factor. 
Each of these parts is monochromatic or rainbow. Since  $a= (r-1)^2 + r$, there are 
 either at least $r$ parts that are rainbow or  at least $(r-1)^2+1$ parts that are monochromatic. 
 If there are $r$ rainbow parts,  the copy of $F$ on these parts  contains a rainbow edge as $F$ is part-rainbow-forced.
 So, assume that  there are at least $(r-1)^2+1$ monochromatic parts. If there are $r$ of those that are of the same color, any edge in a copy of $F$ on these parts is monochromatic. Otherwise there are at most $(r-1)$ parts of each given color, so there are  $r$ monochromatic parts of distinct colors. These 
$r$ parts in turn contain an edge of $F$, and since an edge has at most one vertex from each part, this edge is rainbow.\\

Now, we  verify that the girth of $\cM$ is at least $g$ by an argument similar to one of Lemma~\ref{Lem2}.
To do that, we shall prove by induction on $j$, that $\cM_j$ has girth at least $g$, $j=1, \ldots, a$.
Since $\cM_1$ is a complete $a$-partite $F$ factor, it has girth equal to the girth of $F$, that is at least  $g$.
Assume that $\cM_j$ has girth at least $g$. Let's prove that $\cM_{j+1}$ has girth at least $g$.  
   Recall that $\cM_{j+1} = \cM_j \star_j \cH_j$, i.e., $\cM_{j+1}$ is obtained 
 by amalgamating copies of $\cM_j$ along $\cH_j = H(|\cM_j|_j, g-1)$.  Let $X$ be the $j^{\rm th}$ part of $\cM_{j+1}$, i.e., the vertex set of $\cH_j$.
 Consider a shortest cycle $C$ in $\cM_{j+1}$.  If $C$ is a subgraph of one of these copies of $\cM_j$, then by induction $C$ has length at least $g$.  
  If the  edges of $C$  come from at least two distinct  copies of  $\cM_j$,  then $C$ is  an edge-disjoint  union of hyperpaths   $P_0, P_1, \ldots, P_{m-1}$, each with at least $2$ edges,  from different copies of $\cM_j$, such that the consecutive paths  share a vertex in $X$, i.e., $V(P_i)\cap V(P_{i+1}) = \{u_i\}$, $i=0, \ldots, m-1$, and 
$u_0, \ldots, u_{m-1}$ are distinct vertices from $X$, addition modulo $m$. 
Thus $u_{i}$ and $u_{i+1}$ belong to the same copy of  $\cM_j$ and thus  correspond to the vertices from the same edge of $\cH_j$, $i=0, \ldots, m-1$, addition modulo $m$.
 We see that these  edges of $\cH_j$ form a cycle in $\cH_j$ of length at most half the length of $C$. 
 On the other hand, we know that any cycle in $\cH_j$ has length at least $g-1$, implying that  $C$ has length at least $2(g-1)\geq g$.    
 This concludes the proof of Theorem \ref{main} using an explicit construction. 
 \end{proof}

\section{Proof of Theorem \ref{main} - Probabilistic Construction}\label{probabilistic}

This proof is just a slight generalization of the probabilistic construction for high-girth, high-chromatic-number hypergraphs by Ne\v{s}et\v{r}il  and  R\"odl.
Let an  $\ell$-cycle be a cycle of length $\ell$. 
Let $r, g$ be fixed, put $R = (r-1)^2+1$ and consider an $R$-uniform hypergraph $\cH= \mathcal{H}(n,R,g) =(X, \mathcal{E})$ with  $n$ vertices, girth at least $ g$, and with $|\mathcal{E}|=\lceil n^{1+\frac{1}{g}}\rceil$. Such a graph exists, if $n$ is large enough by Lemma \ref{girth-edges}, see Appendix. 
 
Let's order the hyperedges of $\cH$ as $E_1, E_2, \ldots, E_m$. 
Let $\cM_n$ be the family of all  sequences   $(E_1',\ldots, E_m')$ such that   $|E_i'|=r$ and $E_i'\subseteq E_i$, $i=1, \ldots, m$. For a given sequence $Q \in \cM_n$, let $\cH_Q$ be a hypergraph whose  hyperedges are elements of $Q$.  We say that a coloring of $X$ is {\it good} for $Q$ if there are no monochromatic and no rainbow edges under this coloring of $\cH_Q$.  We say that $Q$ is colorable if there is  a coloring of $X$ that is good for $Q$. We shall count the number of colorable sequences and shall show that it is strictly less than the number of all sequences in $\cM_n$. This will imply that there is a non-colorable sequence corresponding to an rm-unavoidable hypergraph.

Each  hypergraph $\cH_Q$, $Q\in \cM_n$  has girth  at least $ g$ since $\mathcal{H}$ has this property.  In addition $|\mathcal{M}_n|\ge a^{n^{1+\frac{1}{g}}}$, where $a = \binom{R}{r}$, since there are $a$ ways to choose an $r$-element subset from an edge of $\cH$ and  $m\geq  n^{1+\frac{1}{g}}$.
Now we consider a coloring of $X$ with arbitrary number of  colors. Each edge $E$ of $\mathcal{H}$ is colored with at least $r$ or  less than $r$ colors. 
If $E$ is colored with less than $r$ colors, there are $r$ vertices in $E$ of the same color since  $E$ has $R = (r-1)^{2}+1$ elements and $\frac{R}{(r-1)}>(r-1)$. 
If $E$ is colored with at least $r$ colors, there are $r$ vertices with pairwise distinct colors. Thus each edge $E$ of $\cH$  contains a "bad" subset that is either monochromatic or rainbow, and only at most $\binom{|E|}{r} -1 = \binom{R}{r}-1 = a-1$  of all  $r$-element subsets of $E$ could be "good".  Therefore each coloring $c$ of $X$ is good for at most $(a-1)^{\lceil n^{1+\frac{1}{g}}\rceil} \leq (a-1)^{ 1+n^{1+\frac{1}{g} }} $ members of  $\mathcal{M}_n$. Since the total number of colors in $X$ is at most $n$ in any coloring, it is enough to consider colorings with colors ${1,\dots,n}$. Since there are $n^n$ colorings with $n$ colors  we have that 
\begin{eqnarray*}
|\{Q \in \mathcal{M}_n| ~Q \text{ is colorable} \}|
&=&|\bigcup_{c:X\rightarrow [n]}{\bigcup_{Q \in \mathcal{M}_n} \{Q |~ c \text{ is good for } Q\}}| \\
&\le& \sum_{c:X\rightarrow [n]}{|\bigcup_{Q \in \mathcal{M}_n} \{Q |~ c \text{ is good for } Q\}|} \\
&\le& n^n \cdot (a-1)^{1+ n^{1+\frac{1}{g}}}.
\end{eqnarray*}

Next we shall show that    $n^n \cdot (a-1)^{1+ n^{1+\frac{1}{g}}}< a^{ n^{1+\frac{1}{g}}}$ for all sufficiently large $n$.  
Indeed, $n^n (a-1)^{1+ n^{1+\frac{1}{g}}}< a^{ n^{1+\frac{1}{g}}} ~\Leftrightarrow~
  n \ln(n)  + \ln(a-1) < n^{1+\frac{1}{g}}    \ln   \left(  {\frac{a}{a-1}}   \right).$
The last inequality holds  since $\ln({\frac{a}{a-1}})>0$.
Therefore the number of colorable members from  $\cM_n$ is less than the total number of members in  $\cM_n$ and thus there is an non-colorable $Q\in \cM_n$ that gives  $\cH_Q$, an $r$-uniform hypergraph of girth at least $g$ that is rm-unavoidable. \qed

\bibliographystyle{plain}
\bibliography{orderedLiterature.bib}

\section{Appendix}

\begin{lemma}
For any $\ell, g \geq 2$, $q\geq 1$ there is an $\ell$-uniform hypergraph of girth at least $g$ and minimum degree  at least $q$.
\end{lemma}

\begin{proof}
To see that such a hypergraph exists, consider an $\ell$-uniform hypergraph $F$ of girth at least $g$ and chromatic number greater than  $q$. 
If $F$ has a vertex $v$ that belongs to at most $q-1$ edges, delete it from $F$. We obtain a hypergraph $F-v$ of chromatic number greater than  $q$ again because otherwise we can take a proper coloring of $F-v$ with at most $q$ colors and extend it to a proper coloring of $F$. Indeed, if $E_1, \ldots, E_{q'}$, $q'\leq q-1$ are the edges incident to $v$, choose a color for $v$ that is not a color of monochromatic $E_i- v$ under the proper coloring of $F-v$, $i=1, \ldots, q'$, if such a monochromatic edge exists.  Since only at most $q-1$ colors are forbidden for $v$, one color is still available.
Continue this deletion process until possible. The process must stop with a non-empty graph of chromatic number greater than $q$ and minimum degree at least $q$. 
Since it is a sub-hypergraph of the original hypergraph, it has girth at least $g$. 
\end{proof}

\begin{lemma} [\cite{nese78}] \label{number-cycles}
 Let $C(r,\ell,n)$ be the number of $\ell$-cycles in the $r$-uniform complete hypergraph on $n$ vertices, $~ r \ge 3 $. Then 
$C(r,\ell,n)\leq  c(r, \ell)\binom{n}{(r-1)\ell},$ for a function $c(r,\ell)$ independent of $n$.
\end{lemma}

\begin{proof}
Observe that  the largest  number of vertices in an $\ell$-cycle $C$ of length $\ell$ is $(r-1)\ell$. 
Indeed a cycle $C$ of length $\ell$ is defined as a subhypergraph $C$ with $\ell$ distinct vertices $x_0, \dots, x_{\ell-1}$,  $\ell \ge 2$  and distinct hyperedges $E_0, \dots E_{\ell-1}$ such that $x_i, x_{i+1} \in E_i, i = 0, \dots, \ell-1$, addition of indices modulo $\ell$. 
Thus, each hyperedge $E_i$, $i=0, \ldots, \ell-1$,  has at most $r-2$ vertices not in the set $\{x_0, \ldots, x_{\ell-1}\}$. 
Therefore the total number of vertices in $C$ is at most $\ell(r-2) + |\{x_0, \ldots, x_{\ell-1}\}| = \ell(r-2) + \ell = \ell(r-1)$.
Thus, an upper bound on the number of all  $\ell$-cycles is  $\binom{n}{\ell(r-1)}\cdot c(r,\ell)$, where  $\binom{n}{\ell(r-1)}$  is the number of ways
to choose a set on $\ell(r-1)$ vertices and $c(r,\ell)$ is the number of $\ell$-cycles on a given set of $\ell(r-1)$ vertices.
\end{proof}

\begin{lemma} [\cite{nese78}]\label{girth-edges}
For any positive integers $r$ and $s$, $r \ge 2, s \ge 3$  there exists an $n_0 \in \mathbb{N}$ such that for any  $n\ge n_0, \in \mathbb{N}$ there exists an $r$-uniform hypergraph $(X, \mathcal{E})$ with girth at least $s$ and with $|\mathcal{E}|> n^{1+\frac{1}{s}}$.
\label{lem:highgirth+edges}
\end{lemma}

\begin{proof}
We consider a set $\cM= \mathcal{M}(n, r, s)$ of all $r$-uniform  hypergraphs on vertex set $[n]$ with  
 $m=2\lceil n^{1+\frac{1}{s}}\rceil$ edges.
Then $|\cM| = \binom{\binom{n}{r}}{m}$. 
Choose a hypergraph  $\cH$ from  $\cM$  randomly and uniformly, i.e., with probability $\frac{1}{|\cM|}$.
Let $K$ be a complete $r$-uniform hypergraph on vertex set $[n]$. 
Call cycles  of length smaller than $s$ bad. Let $X_j$ be the  number of  cycles of length $j$ in $\cH$ and $X_{bad}$ be the \ number of bad cycles.
Then ${\rm Exp}(X_{j}) = \sum _{C}   {\rm Prob} (C\subseteq \cH)$, where the sum is over all cycles $C$ of length $j$ in $K$. 
Then ${\rm Exp}(X_{j}) \leq  C(r, j, n)  {\frac{\binom{\binom{n}{r}-j}{m-j}}{\binom{\binom{n}{r}}{m}}}$, where $C(r,j,n)$ is the number of cycles of length $j$ in $K$ and 
second term is the probability of occurrence of such a cycle. Using Lemma \ref{number-cycles}, we have that 
${\rm Exp}(X_{j}) \leq  c(r, j) \binom {n}{(r-1)j} {\frac{\binom{\binom{n}{r}-j}{m-j}}{\binom{\binom{n}{r}}{m}}}$. Then, for constants $\widetilde{c}(r,j)$, $j=2, \ldots, s-2$ and $\widetilde{C}(r,s)$, we have

\begin{align*}
{\rm Exp}(X_{bad })&= \sum_{j=2}^{s-1}  {\rm Exp}(X_{j})\\
                &\leq  \sum_{j=2}^{s-1}{c(r,j)\cdot \binom{n}{(r-1)j}{\frac{\binom{\binom{n}{r}-j}{m-j}}{\binom{\binom{n}{r}}{m}}}}              \\  
               &= \sum_{j=2}^{s-1}{c(r,j)\cdot \binom{n}{(r-1)j}{\frac{m\cdot(m-1)\cdots(m-j+1)}{\binom{n}{r}\cdot(\binom{n}{r}-1)\cdots(\binom{n}{r}-j+1)}}}\\
              & \leq  \sum_{j=2}^{s-1}  c(r,j) \cdot  \binom{n}{(r-1)j}   \left(  \frac{m}{\binom{n}{r}}   \right)^j\\             
                &\le  \sum_{j=2}^{s-1}   \widetilde{c}(r,j)  n^{(r-1)j - rj}  m^j \\
                &\le  \sum_{j=2}^{s-1}   \widetilde{c}(r,j)  n^{(r-1)j - rj} n^{(1 +1/s)j} \\
               & \leq  \widetilde{C}(r,s)  n.
                 \end{align*}  
                 
  Since $      {\rm Exp}(X_{bad })\leq     \widetilde{C}(r,s)  n$, there is a hypergraph from $\cM$ with at most       $\widetilde{C}(r,s)  n$ cycles of length at most $s-1$. 
  Delete an edge from each such cycle and obtain a hypergraph on at least $2 n^{1+1/s} - \widetilde{C}(r,s)  n >  n^{1+1/s}$ edges and girth at least $s$.
\end{proof}

\end{document}